\documentclass{amsart}
\usepackage[utf8]{inputenc}

\usepackage[english]{babel}

\usepackage{amssymb}
\usepackage{amsmath}
\usepackage{amsmath,amscd}
\usepackage{amsthm}
\usepackage{enumerate}
\usepackage[nocompress]{cite}
\usepackage{xcolor}

\makeindex

\title{On the study of decompositions of forms in four variables}

\newcommand{\C}{\mathbb{C}}

\newcommand{\Pj}{\mathbb{P}}

\newtheorem{theorem}{Theorem}[section]
\newtheorem{proposition}[theorem]{Proposition}

\theoremstyle{definition}

\newtheorem{example}[theorem]{Example}

\newtheorem{remark}[theorem]{Remark}


\author{Luca Chiantini}
\address{Dipartimento di Ingegneria dell'Informazione e Scienze matematiche, Universit\'a di Siena (Italy)}
\email{luca.chiantini @unisi.it}
\thanks{Al mio amico Giorgio, con cui ho festeggiato le vittorie della Lupa nel 2016}

\begin{document}
\abstract{In the space of sextic forms in $4$ variables with a decomposition of length $18$ we determine and describe a closed subvariety which contains
all non-identifiable sextics. The description of the subvariety is geometric, but one can derive from that an algorithm which can guarantee that a given form is identifiable.}

}

\maketitle

\section{Introduction}

The paper describes an application of geometric tools, mainly from the theory
of finite sets in projective spaces, to the study of Waring decompositions of  forms.\\
The tools have been introduced and employed, in a series of papers, mainly for forms of degree $4$
or for forms in three variables (see  \cite{AngeCVan18}, \cite{AngeC20}, \cite{AngeC22}, \cite{AngeC}, \cite{COtt}).
Since quaternary forms of degree $5$  are considered in a forthcoming paper \cite{BCDMP},
we turn now our attention to forms of degree $6$ in four variables.\\
Our starting point is the same starting point of the celebrated Kruskal's criterion for the minimality and
uniqueness of a decomposition (to be precise, in its version for symmetric tensors). 
We assume that we know a (Waring) expression
of a form $F$ in terms of powers of linear forms,  as the one given in formula \eqref{expr} below. 
The problem consists of determining if the expression is minimal, in which case it
computes the Waring rank of F. In addition, one would like to know if the expression
is unique (up to trivialities). \\
We attack the problem by considering the linear forms appearing in the expression as a set of points $A$ in a projective space $\Pj^3$, and 
analyzing the existence of another set $B$, of length smaller or equal than the length of $A$, whose
$6-$Veronese image spans $F$.\\
It turns out that the union $Z=A\cup B$ must satisfy several geometric and algebraic restrictions. This makes it
possible to analyze the situation up to rank $18$. Indeed, we prove that when the length (= cardinality) of $A$ is strictly smaller than $18$
and $A$ is sufficiently general in a very precise sense (see the statement of Proposition \ref{base16}), then 
the expression is necessarily unique. The geometric situation in this case is similar to the one treated in Kruskal's criterion
(which, by the way, even in its reshaped version described in \cite{COttVan17b}, cannot work for $r>14$ in the case of quaternary sextics).\\
The case $r=18$ turns out to be different. For $r=18$, even if $A$ is completely general, there are forms
in the span of the $6-$Veronese image of $A$ for which a second decomposition $B$ exists. We can be more specific: 
when $A$ is general, so that cubic surfaces through $A$ define a complete intersection irreducible curve $C$ of degree $9$,
 then $B$ is forced to be residual to $A$ in a complete intersection of $C$ and a quartic surface. This allows us to parameterize 
 the possible sets $B$, and thus parameterize
 a (locally closed) subvariety $\Gamma$ of the span of $v_6(A)$, which contains the forms $F$ of degree $6$ in $4$ variables, rank $18$,
 which are not identifiable. The closure of $\Gamma$ is the image of a map from a subspace of the projective 
 space $\Pj((I_A)_4)$ to $\langle v_6(A)\rangle$. We refer to Theorem \ref{main} for a more precise description.\\
 In particular, we get that if $F$ is a non-identifiable form, then the second decomposition $B$ is bounded to
 an invariant curve $C$, defined by $A$. This is a case of {\it confinement} for decompositions of forms, as described in general
 in \cite{AngeBocciC18}.\\  
 Since the generic rank of a form of degree $6$ in four variables is $21$, one may wonder what happens for the missing cases  $r=19,20,21$.
 For $r=19$, the same procedure proves that a hypothetical second decomposition $B$ must be bounded to the unique cubic surface defined by
 $A$, but we are not able to characterize it any more. For $r=20,21$ we have no precise characterization. This is probably
 due to the fact that the theory of finite sets in $\Pj^3$ is far from being completely understood, and also opens a 
 series of questions on the structure of finite sets in higher dimensional spaces, which could suggest directions to
 investigators in the field.
 
 \section {Preliminaries}
 
 All polynomials in the paper are defined over the complex field. \\
 We will often, by abuse,
 use the same letter to indicate a form in a polynomial ring, the projective hypersurface defined by the form, and
 the point defined by the form in the corresponding projective space.
 \smallskip
 
 Given a finite set $A$ in a projective space, we denote by $\ell(A)$ its  length (i.e. its cardinality).
\smallskip
 
 Consider a form $F$ of degree $6$ in $4$ variables, over the complex field.
 
 Assume we know a Waring expression of $F$ (of length $r$) as a linear combination of powers of linear forms
\begin{equation}\label{expr} F = \sum_{i=1}^r a_i L_i^6 \end{equation}
 but we do not know a priori if the expression is minimal or unique (up to trivialities). Thus we do not know if $r$
 is the (Waring) rank of $F$, and we do not  know whether $F$ is identifiable or not.
 \smallskip
 
 On the other hand, we can certainly assume that the expression is {\it non-redundant}, in the sense that
 the powers $L_i^6$'s are linearly independent and no coefficient $a_i$ is $0$. 
 \smallskip
 
 Call $A=\{L_1,\dots, L_r\}$ the set of linear forms involved in the expression, 
 considered as points in a projective space $\Pj^3$.
 If we denote with $v_d:\Pj^3\to \Pj^N$ the $d$-Veronese map, the expression tells us 
 that $F$ (as we said above, identified by abuse with one point
 of the space $\Pj^{83}$ of sextic forms in $\Pj^3$) belongs to the span of the Veronese image
 $v_6(A)$. The non-redundancy of $A$ is equivalent to saying that, for all {\it proper} subsets $A'\subset A$,
$F$ is not contained in the span of $v_6(A')$.\\
 
 We have full control on the set $A$, so we may assume that we know all its invariants.
 Thus we can assume that
  
 $$ (*) \quad A \mbox{ is in } General\ Position  \ (GP)$$
 which, in this setting, means that all subsets  of $A$ have maximal Hilbert function.\\
 Notice that if $A$ has this property, then all subsets of $A$ also have it.
 \smallskip

 \begin{remark}\label{Kru} When $r\leq 14$, then the celebrated Kruskal's criterion, in its reshaped version
 (see \cite{COttVan17b}) guarantees that $r$ is the rank of $F$, and the expression is unique (up to trivialities: product by a scalar or reordering).\\ 
 Namely, if $u=\min\{r,10\}$ then necessarily
 $$ r\leq \frac{u+u+u-2}2,$$
 thus we can take a partition $6=2+2+2$ and consider $F$ as a tensor of 
 $Sym^2(\C^4)\otimes Sym^2(\C^4)\otimes Sym^2(\C^4) $. Since the 
 second Kruskal's rank of $A$ is $u$ by the genericity assumption, then a direct application
 of Kruskal's criterion guarantees that \eqref{expr} is the unique expression of $F$ of length $r$.
 \end{remark}
 
 When $r>14$, we assume the existence of another expression
 \begin{equation}\label{expr2} F = \sum_{i=1}^s a_i M_i^6, \quad s\leq r \end{equation}
 and call $B=\{M_1,\dots,M_s\}$ the consequent finite set in $\Pj^3$.
 
 Again we may directly assume that also $B$ is non-redundant.

 When $r\geq 15$ the Kruskal's criterion cannot provide a proof of the minimality and uniqueness of the expression \eqref{expr}.
 Indeed in this case new expressions are possible. 
 A finer geometrical analysis is required to understand the situation.
 \smallskip
 
 Call $h_A$, $h_B$, $h_Z$ the Hilbert functions of $A,B$ and $Z=A\cup B$ respectively.\\
 By assumptions we know that the difference $Dh_A(i)=h_A(i)-h_A(i-1)$ is defined by the following table

\begin{center}\begin{tabular}{c|ccccccc}
$i$ &             $0$ & $1$ & $2$ &   $3$ &   $4$ &   $5$  & $ \dots $ \\  \hline
$Dh_A(j)$ & $1$ & $3$ &   $6$ &   $r-10$ &   $\max\{0,r-20\}$ & $0$  & $ \cdots $ \cr
\end{tabular}\end{center}

From \cite{AngeC20} Proposition 2.19, we know that 
$$\dim(\langle v_6(A)\rangle\cap \langle v_6(B)\rangle) = \ell(A \cap B) - 1 + h^1_Z(6).$$
where $h^1_Z(i)$ is defined by $h^1_Z(i)=\ell(Z)-h_Z(i)$.\\
In particular $h_Z(6)<\ell(Z)$ when $A,B$ are disjoint.
\smallskip

We recall the Cayley-Bacharach property of $Z$ from \cite{AngeCVan18} and \cite {AngeC20}, Section 2.4.

\begin{remark}\label{CB}  Since $A,B$ are both non-redundant, if $A\cap B=\emptyset$ then the set $Z$ satisfies
the \emph{Cayley-Bacharach property}. In particular for $j=0,1,2,3$,
$$ \sum_{i=0}^j Dh_Z(i) \leq \sum_{i=0}^j Dh_Z(7-j).$$
\end{remark}

\begin{proposition}\label{minim} Assume $r\leq 20$. Then $s=\ell(B)\geq r$. If $r=15$ then $A,B$ are disjoint.
Moreover, for all $r$ the ideals of $A$ and $Z$ agree up to degree $3$.
\end{proposition}
\begin{proof} If $A\cap B=\emptyset$, then by Remark \ref{CB} we must have:
$$ \ell(Z)=\ell(A)+\ell(B)\geq\sum_{i=0}^7 Dh_Z(i)\geq 2\sum_{i=0}^3 Dh_Z(i)\geq 2\sum_{i=0}^3 Dh_A(i)=2\ell(A)$$
which proves $s\geq r$. If $r=s$, the  inequalities become equalities, and this implies the result on the ideals of $A$ and $Z$.\\
Assume $A\cap B\neq \emptyset$, i.e. assume $L_i=M_i$ for 
 $i=1,\dots,j$, for some $j>0$. \\
 Then 
  \begin{align*}
 F &= a_1L_1^6+\dots +a_jL_j^6+a_{j+1}L_{j+1}^6+\dots +a_rL_r^6 \\
 &= b_1L_1^6+\dots +b_jL_j^6+b_{j+1}M_{j+1}^6+\dots +b_rM_r^6.
 \end{align*}
Define $F'$ by
  \begin{align*}
 F' = (a_1-b_1)L_1^6+\dots +(a_j-b_j)L_j^6 &+a_{j+1}L_{j+1}^6+\dots +a_rL_r^6 \\
 &=b_{j+1}M_{j+1}^6+\dots +b_sM_s^6
 \end{align*}
 $F'$ has two disjoint decomposition. The former can have some vanishing coefficients, but its length
  in any case is at least $r-j$, while the latter has length $\leq s-j$.\\
  If $r=15$ we obtain a contradiction by the reshaped Kruskal's criterion (Remark \ref{Kru}) or by what we concluded above in
  the disjoint case.  Then, arguing by induction on $r$, we get that $s\geq r$. \\
  If $A_0,B_0$ are the two decompositions of $F'$ defined above, then by induction 
   the ideals of $A_0$ and $Z_0=A_0\cup B_0$ agree
  up to degree $3$.  Since $A,B$ are obtained from $A_0,B_0$ by adding the same subset $S$, then also the
  ideals of $A$ and $Z$ agree up to degree $3$.
 \end{proof}
 
The  minimality of the expression \eqref{expr} proved in the previous result indeed also follows
from Theorem 1.2 of  \cite{Ball19}, or by \cite{MourOneto20} Theorem 3.1.

\section{The case $r=15$}

We know from Proposition \ref{minim} and its proof  that if $F$ has two decompositions $A,B$, then $A\cap B=\emptyset$. \\
We show an example in which  the second decomposition $B$  exists.
 
 \begin{example}\label{ell5} Assume that $A$ is a general set of $15$ points in a general elliptic quintic curve $C$.
 The $6$-Veronese map maps $C$ to a normal elliptic curve of degree $30$ which spans a $\Pj^{29}$. In $\Pj^{29}$
a general point has two different decompositions with respect to the elliptic curve $C$ (see Proposition 5.2  of \cite{CCi06}). Thus one gets that a
general $F$ in the span of $v_6(A)$ has exactly two different decompositions.\\
 It is easy indeed to construct examples of forms $F$ with two decompositions of this type. A general set $A$
 of $15$ points in an elliptic quintic and a general $F$ in the span of $v_6(A)$ will do.
 \end{example}

On the other hand, it is also simple to realize that a general set $A$ of $15$ points in $\Pj^3$ does not lie in an elliptic quintic.
This is just a count of parameters: the Hilbert scheme of elliptic quintics has dimension $5\cdot 4 =20$, 
so the sets of $15$ points in such curves cannot depend on more than $20+15=35$ parameters; on the other hand, the family of sets of $15$
points in $\Pj^3$ has dimension $45$.
\smallskip

One can easily exclude that a given set $A$ of $15$ points in $\Pj^3$ lies in an elliptic quintic by
considering the base locus of the system of cubics through $A$ which, by assumption, has dimension $5$.

\begin{proposition}\label{base15} Assume $r=15$ and assume that the base locus of the system of cubics 
through $A$ contains no curves. Then $A$ is the unique minimal decomposition of $F$. 
\end{proposition} 
\begin{proof} Assume there exists a second decomposition $B$ of length $\leq 15$. Arguing as in the final part of the 
proof of Proposition \ref{minim}, since we can apply the reshaped Kruskal's criterion for decompositions of
length $\leq 14$, we see that $A,B$ must be disjoint. We know that the ideal of $Z=A\cup B$ coincides
with the ideal of $A$ in degree $3$. Since the base locus of the system of cubics through $A$ contains no curves, then
by B\'ezout $Z$ has length at most $27$. Thus $\ell(B)\leq 12$, which is excluded by Proposition \ref{minim}.
\end{proof}

One checks easily the dimension of the base locus of the system of cubics through $A$, by standard
computer algebra packages.

 \section{The cases $r=16,17$}

 The situation for $r=16,17$ is quite similar to the case $r=15$, except that now an intersection between the two decompositions
 is allowed.
 
 \begin{example}\label{nodisj} Let $A_0$ be a general set of $15$ points lying in a general elliptic quintic curve $C$. We saw in Example \ref{ell5}
 that a general form $F_0$ in the span of $v_6(A_0)$ has a second decomposition $B_0$ of length $15$, disjoint from $A_0$.
 If $L_0$ is a general linear form, then $\{L_0\}\cup A_0$ and $\{L_0\}\cup B_0$ are two different, non-disjoint, decompositions of 
 length $16$ of $L_0^6+F_0$.\\
 Arguing as in Proposition \ref{minim}, one sees that these two decompositions are minimal, when $A_0,B_0,L_0$ are general. \\
 \end{example}
 
 Also examples with different disjoint decompositions are possible.
 
 \begin{example}\label{rat5} Let $A$ be a general set of $16$ points lying in a general rational quintic curve $C$. By B\'ezout, since $C$ is irreducible,
 the ideal of $C$ and the ideal of $A$ agree in degree $3$. The Veronese map $v_6$ maps $C$ to $\Pj^{30}$. Since no curves are defective, a general
 point $F$ of $\Pj^{30}$ has infinitely many (mostly disjoint) decompositions of length $16$ with respect to $v_6(C)$.\\
 Sets $A$ of this type lie in the \emph{Terracini locus}, as defined in \cite{BallC21}: the differential of the map from the abstract $16$-secant variety
 to the space $\Pj^{83}$ of $v_6(\Pj^3)$ drops rank over a general $F\in\langle v_6(A)\rangle$.  
 \end{example}
 
 \begin{example} Starting with forms with two decompositions of length $16$, as e.g. in Example \ref{rat5}, and adding one point as in Example 
 \ref{nodisj}, one finds easily examples of non-disjoint different decompositions of length $17$ for some sextics $F$.
 \end{example}
 
 As in the case $r=15$, if the system of cubics through $A$ has no curves in the base locus, then the decomposition $A$ of $F$ is unique.

 \begin{proposition}\label{base16} Assume $r=16$ or $r=17$ and assume that the base locus of the system of cubics 
through $A$ contains no curves. Then $A$ is the unique minimal decomposition of $F$. 
\end{proposition} 
\begin{proof} The proof  is given only for $r=16$, since the other case is completely analogous.\\
Assume there exists a second decomposition $B$ of length $16$. If $A\cap B=\emptyset$, since the ideal of $Z=A\cup B$ coincides
with the ideal of $A$ in degree $3$,  by B\'ezout $Z$ has length at most $27$. Thus $\ell(B)\leq 11$, which is excluded by Proposition \ref{minim}.\\
If the intersection $A\cap B$ contains $j>0$ points, then as above write
  \begin{align*}
 F &= a_1L_1^6+\dots +a_jL_j^6+a_{j+1}L_{j+1}^6+\dots +a_{16}L_{16}^6 \\
 &= b_1L_1^6+\dots +b_jL_j^6+b_{j+1}M_{j+1}^6+\dots +b_{16}M_{16}^6.
 \end{align*}
Define $F'$ by
  \begin{align*}
 F' = (a_1-b_1)L_1^6+\dots +(a_j-b_j)L_j^6 &+a_{j+1}L_{j+1}^6+\dots +a_{16}L_{16}^6 \\
 &=b_{j+1}M_{j+1}^6+\dots +b_{16}M_{16}^6
 \end{align*}
 $F'$ has two disjoint decompositions, one for which $A'$ is contained in $A$. Thus the system of cubics through $A'$
 has no curves in the base locus. Even if the length of $A'$ is $15$, we have a contradiction with Proposition \ref{base15}.
\end{proof}
 
 Since for $r\leq 17$ and $A$ very general the system of cubics through $A$ has no curves in the base locus, the previous proposition excludes
 the existence of a second decomposition, except for sets $A$ contained in a Zariski closed subset of $(\Pj^3)^r$.

  \section{The case $r=18$}
  
  For $r=18$ and $A$ general, the base locus of the system of cubics through $A$ is a complete intersection
  curve $C$ of degree $9$ and genus $10$. There is no way to use a strategy similar to the statement of
  Proposition \ref{base16} in order to prove the identifiability of $F$.
  
  \begin{remark}\label{singul}
From Proposition \ref{base15} and Proposition \ref{base16} it turns out that, when $r=15,16,17$ and the
 system of cubics through $A$ has no curves in the base locus, then all forms $F$ in the span
 of $v_6(A)$ are identifiable of (Waring) rank $r$, unless   the decomposition $A$ is redundant for $F$, i.e. unless $F$ sits in the
 span of some strict subset of $v_6(A)$. 
\end{remark}

We can see immediately that the situation changes completely for $r=18$.

\begin{example}\label{18} Let $A$ be a \emph{general} set of $18$ points in $\Pj^3$. Then $A$ is contained in the complete intersection
of two cubics $G_1,G_2$. Consider the complete intersection curve $C=G_1\cap G_2$ and let $G$ be a general quartic not
containing $C$. \
The intersection of $C$ with the surface $G$ consists of $36$ points $Z=A\cup B$. $B$ is thus  a set of $18$ points in the curve $C$, disjoint from $A$.
By the Cayley-Bacharach property of complete intersections, one knows that $h^1_Z(6)> 0$. Thus by Proposition 2.19 of
\cite{AngeC20}, we know that $\langle v_6(A)\rangle$ and $\langle v_6(B)\rangle$ meet in some point $F$.
Such $F\in \langle v_6(A)\rangle$ has a second decomposition $B$ of length $18$.
\end{example}

\begin{remark}\label{angec}
By Proposition 3.9 of \cite{AngeC22}, when $A,B$ are disjoint decompositions of $F$, then the sum of the homogeneous ideals $I_A+I_B$
does not coincide with the polynomial ring $R$ in degree $6$, and $F$ is dual to  $I_A+I_B$
\end{remark}

Consider again the sets $A,B$ described in Example \ref{18}.\\
The ideal of $B$ can be found from $G$ and the ideal of $A$ as a result of the mapping cone process (see \cite{Ferrand75}).
By the Minimal Resolution Conjecture, which holds in $\Pj^3$ (see \cite{BallGer86}), a resolution of the ideal $I_A$ is given by 
$ 0\to R^8(-6)\to R^{18}(-5)\to R^2(-3)\oplus R^9(-4)\to I_A\to 0$. Combining with the Koszul complex of $G_1,G_2,G$ one obtains a diagram
 \begin{equation}\label{dia} \begin{matrix}
 0  & \to  & R(-10) & \stackrel{\alpha'}\rightarrow & R(-6)\oplus R^2(-7) & \stackrel{\beta'}\rightarrow  & R^2(-3)\oplus R(-4) &  \to  & I_Z  &\to & 0 \\
   & &   \gamma\downarrow & &        \gamma'\downarrow & &  \gamma''\downarrow &  &  & & \\
 0  & \to  & R^8(-6) & \stackrel{\alpha}\rightarrow & R^{18}(-5)  &  \stackrel{\beta}\rightarrow  & R^2(-3)\oplus R^9(-4) &\to  & I_A & \to & 0
 \end{matrix}\end{equation}
 where the map $\gamma''$ is defined by $G_1,G_2,G$.
From the diagram one obtains a resolution of $I_B$ by the dual of the mapping cone:
$$  0\to R^8(-6)\to R^{18}(-5)\stackrel{(\alpha\oplus\gamma')^\vee}\longrightarrow R^2(-3)\oplus  R^9(-4) \stackrel{(\alpha'\oplus\gamma)^\vee}\longrightarrow I_B\to 0 $$
Thus there is a standard way to compute $I_B$, hence $I_A+I_B$, from $I_A$ and $G$.
\medskip

We have then all the ingredients to study the existence of a second decomposition for $F$.

\begin{proposition}\label{disj18} Assume that the decomposition $A$ of length $18$ of $F$, satisfying condition $(*)$, also satisfies the following condition: 
for all subsets $A'\subset A$ of length $17$, the linear system of cubics through $A'$ has base locus of dimension $0$. Then any
other decompositions of length $18$ of $F$ is disjoint from $A$.
\end{proposition}
\begin{proof} Assume there exists another decomposition $B$ of length $18$ with $\ell(A\cap B)=j>0$. Then arguing as in the proof of 
Proposition \ref{minim} one finds another sextic form $F'$ with decompositions
  \begin{align*}
 F' = (a_1-b_1)L_1^6+\dots +(a_j-b_j)L_j^6 &+a_{j+1}L_{j+1}^6+\dots +a_{18}L_{18}^6 \\
 &=b_{j+1}M_{j+1}^6+\dots +b_{18}M_{18}^6,
 \end{align*}
 where $A=\{L_1,\dots,L_{18}\}$ and $a_i,b_i\neq 0$ for all $i$. If some coefficient  $a_i-b_i$, $i=1,\dots,j$, is non-zero,
 then the second decomposition of $F'$ has length smaller than the first one, which is contained in $A$. We get a contradiction 
 with Proposition \ref{minim}. Thus $a_i=b_i$ for all $i=1,\dots,j$. But then $F'$ has two disjoint decompositions of length $18-j$,
 and one of them $A'=\{L_{j+1},\dots,L_{18}\}$ is contained in $A$. By assumption the system of cubics through $A'$
 has no curves in the base locus. Then we get a contradiction with either the Reshaped Kruskal's Criterion, or
 Proposition \ref{base15}, or Proposition \ref{base16}.
\end{proof}

\begin{theorem}\label{main} Let $F$ be a sextic in 4 variables, with a non-redundant decomposition $A$ of length $18$.
Assume that $A$ satisfies the following properties.
\begin{itemize}
\item[(*)] $A$ is in  General Position;
\item[(**)]  for all subsets $A'\subset A$ of length $17$, the linear system of cubics through $A'$ has base locus of dimension $0$;
\item[(***)]  the base locus of the pencil of cubics through $A$ is an irreducible curve $C$. 
\end{itemize}
Then $A$ is minimal, and any other decomposition $B$ of length $18$ of $F$ (if any)
is disjoint from $A$, and $Z=A\cup B$ is a complete intersection of surfaces of degrees $3,3,4$.
\end{theorem}
\begin{proof} The unique thing that remains to prove is the last assertion, i.e. that $A\cup B$ is the intersection of $C$ with a quartic surface.\\
If $B$ exists, $Z=A\cup B$ lies in the pencil of cubics containing $A$, by Proposition \ref{minim}. If all the quartics containing
$Z$ are composed with the pencil, then $h_Z(4)= 35-8=27$, so that $Dh_Z(4)=9$. But then
$Dh_Z(5)+Dh_Z(6)+Dh_Z(7)\leq  9<Dh_Z(2)+Dh_Z(1)+Dh_Z(0)$, which contradicts the Cayley-Bacharach property.
hence there is a quartic containing $Z$ and not $C$. The claim follows.
\end{proof}

\begin{remark} For a given form $F$ and a decomposition $A$ of length $18$, one can produce a procedure
which tests if $A$ is unique i.e. if $F$ is identifiable of rank $18$, as follows.\\
\begin{enumerate}
\item Control if $A$ is in $GP$.
\item Control that the system of cubics through any subset of length $17$ of $A$ has $0$-dimensional base locus.
\item Control that the system of cubics through $A$ has an irreducible nonic curve as base locus.
\item Consider a linear space $W$ in $(I_A)_4$ orthogonal to the $8$-dimensional subspace spanned by the cubics through $A$.
\item For all $G\in W$ compute the generators of the residue $B$ of $A$ in $G\cap C$, in terms of coordinates of $G\in W$.
\item Prove that for no choice of the coordinates of $G$ the form $F$ is dual to $I_A+I_B$.
\end{enumerate}
Notice that the generators of $I_B$, mod the cubics containing $C$, are $9$ quartics, by the resolution following diagram \ref{dia}.\\
One of the most expensive points in the procedure is step (1), which requires to control that none of the $\binom{18}8=43,758$ subsets of length $10$
of $A$ is contained in quadrics. 
\end{remark}



\end{document}